\documentclass[11pt]{article}
\usepackage{amsmath}
\usepackage{amssymb}
\usepackage{amsfonts}
\usepackage{graphicx}

\newtheorem{theorem}{Theorem}

\newtheorem{corollary}[theorem]{Corollary}

\newtheorem{proposition}[theorem]{Proposition}

\newenvironment{proof}[1][Proof]{\noindent \textbf{#1.} }{\  \rule{0.5em}{0.5em}}

\begin{document}

\title{Explicit formulas for Killing magnetic curves in Heisenberg group%
\thanks{%
The paprer was supported by Laboratory of fundamental and applied
mathematics, University of Oran}}
\author{\textsc{Khadidja Derkaoui}$^{1}$\thanks{%
Corresponding author} and \textsc{Fouzi Hathout}$^{2}$ \\
$^{1}${\small Department of mathematics, university of Chlef, Algeria}\\
{\small Email: derkaouikhdidja248@hotmail.com}\\
$^{2}${\small Department of mathematics, university of Sa\"{\i}da, Algeria}\\
{\small Email: f.hathout@gmail.com}}
\date{}
\maketitle

\begin{abstract}
In this paper, We present the geometry three dimensional Heisenberg group $(%
\mathbb{H}_{3},g)$ and its geodesics curves. After, we study the Killing
magnetic curves and some geodesic Killing magnetic curves with its explicit
formulas for such curves.


\end{abstract}

\section{Introduction}

In physics, particularly in electromagnetism, the trajectory of charged
particles moving under the action of a magnetic fields makes an important
research topic. In geometry, this trajectory in any manifold is known as a
magnetic curve.

On a Riemannian manifold, the magnetic curve $\gamma $ is modeled as a
solution of two order differential equation $\nabla _{\gamma ^{\prime
}}\gamma ^{\prime }=\varphi \left( \gamma ^{\prime }\right) $ known as
\textsc{Lorentz} equation, where $\varphi $ is $(1,1)$-tensor field that
present the \textsc{Lorentz} force associated to the magnetic fields $F$.
The magnetic curve generalise the geodesic curves in following way: When the
particles move under the absence of the magnetic fields so freely only under
the influence of gravity (i.e. $\varphi \equiv 0$), the \textsc{Lorentz}
equation is exactly the geodesic equation $\nabla _{\gamma ^{\prime }}\gamma
^{\prime }=0.$

The magnetic curve was studied intensively in different kind of manifolds in
Riemannian, Lorentzian and generally in pseudo-Riemannian concept. (See \cite%
{br}, \cite{dim}, \cite{dimn}, \cite{o})

Furthermore, if the magnetic fields $F$ corresponds to a killing vector,
then the trajectory of particles moving under the action of $F$ is called
Killing magnetic curve.

Interesting results on Killing magnetic curves are given in Euclidian
3-space (\cite{dm}), Minkowski 3-space (\cite{dm1}), $\mathbb{S}^{2}\times
\mathbb{R}$ (\cite{mn}), $SL(2,\mathbb{R})$ (\cite{e}), Sol Space (\cite{ei2}%
), warped product manifold (\cite{isc}), almost paracontact manifold (\cite%
{cmp}), Almost Cosymplectic Sol Space (\cite{ei}) and in Walker manifolds (%
\cite{bd}).

The framework of the paper is to study the Killing magnetic curves in $3$%
-dimensional Heisenberg group.

The contents of this study is the following: we present in the second
section a general notions and definitions of killing magnetic curves.

The section $3$ is devoted to the geometry of Heisenberg group and its
contact structure.

Finally in the section $4$, we classifies fourth kind of Killing magnetic
curves, we give its explicit parametric formulas at each kind in different
subsections and we close it with figures presented in Euclidian space.

\section{Preliminaries}

On Riemannian manifold $(M,g),$ the magnetic curves is a trajectories of the
charged particles moving under the action of the magnetic fields $F$ which
can be represented by a closed $2$-form
\begin{equation}
F(X,Y)=g(\varphi \left( X\right) ,Y)  \label{0.01}
\end{equation}%
where $X,Y$ are vector fields on $(M,g)$ and $\varphi $ is skew-symmetric $%
(1,1)$-tensor field. that present the \textsc{Lorentz} force associated to $%
F $. For a regular curve $\gamma :I\subset \mathbb{R}\rightarrow (M,g)$, $%
\gamma $ is called magnetic curve if%
\begin{equation}
\nabla _{\mathbf{t}}\mathbf{t=}\varphi \left( \mathbf{t}\right)  \label{0.1}
\end{equation}%
known as the Lorentz equation, where $\nabla $ is the Levi-Civita connection
associated to $g\ $and $\mathbf{t}=\gamma ^{\prime }$ is the speed vector of
$\gamma .$ The magnetic curve generalise the notion of geodesic curve under
arc-length parametrization.

From the property of magnetic curve $\gamma $%
\begin{equation*}
\frac{d}{dt}g(\gamma ^{\prime },\gamma ^{\prime })=0
\end{equation*}%
$\gamma $ has a constant speed vector. In particular, if $\gamma $ is
parameterized by the arc length, then $\gamma $ is called \textit{normal
magnetic curve}.

We call $K$\textbf{\ }a Killing vector field on $M$ if it satisfy the
Killing equation%
\begin{equation*}
g(\nabla _{X}K,Y)+g(\nabla _{Y}K,X)=0
\end{equation*}%
for every vector fields $X,Y$ on $M$.

We define on M the cross product of two vector fields $X,Y$ on $M$ as

\begin{equation}
g(X\times Y,Z)=dvg(X,Y,Z)  \label{0.02}
\end{equation}%
for all vector fields $Z$ on $M$ and $dvg$ denotes the volume form on $M.$

Let $F_{K}=\emph{i}_{K}dvg$ be the Killing magnetic field corresponding to
the Killing vector field $K$ on $M$, where $\emph{i}$ is inner product.
Then, the $(1,1)$-tensor fields corresponding to the Lorentz force of $F_{K}$
is

\begin{equation}
\gamma (X)=K\times X.  \label{0.03}
\end{equation}%
Then, we can rewrite the Lorentz equation Eq.(\ref{0.1}) as%
\begin{equation}
\nabla _{\mathbf{t}}\mathbf{t=}K\times \mathbf{t}  \label{0.2}
\end{equation}%
and the it solution is called \textit{Killing magnetic curves} corresponding
to the killing vector fields $K.$

In the sequel, to simplify, we call this curve a $K$\textit{-magnetic curve.}

\section{Geometry structure of $\mathbb{H}_{3}$}

The Heisenberg group $\mathbb{H}_{3}$ is a quasi-abelian Lie group
diffeomorphic to $\mathbb{R}^{3}$ and it has the standard representation in $%
GL(3,\mathbb{R})$ as
\begin{equation*}
\mathbb{H}_{3}=\left \{ \left(
\begin{array}{ccc}
0 & x & z \\
0 & 0 & y \\
0 & 0 & 0%
\end{array}%
\right) \mid \left( x,y,z\right) \in \mathbb{R}^{3}\right \}
\end{equation*}%
endowed with the multiplication%
\begin{equation*}
(x_{1},y_{1},z_{1})(x_{2},y_{2},z_{2})=(x_{1}+x_{2},y_{1}+y_{2},z_{1}+z_{2}-x_{1}y_{2}).
\end{equation*}%
The invariant Riemannian metric with respect to the left-translations
corresponding to that multiplication is denoted by
\begin{equation}
g=\frac{1}{\lambda ^{2}}dx^{2}+dy^{2}+(xdy+dz)^{2}.  \label{1}
\end{equation}%
where $\lambda $ is a strictly positive real number. All left-invariant
Riemannian metric on the $\mathbb{H}_{3}$ is isometric to the metric $g$.

The Levi-Civita connection $\nabla $ of the metric $g$ with respect to the
left-invariant orthonormal basis%
\begin{equation}
e_{1}=\partial y-x\partial z,\  \ e_{2}=\lambda \partial x,\  \ e_{3}=\partial
z,  \label{2}
\end{equation}%
with dual basis%
\begin{equation*}
\omega ^{1}=dy;\  \  \omega ^{2}=\frac{1}{\lambda }dx;\  \  \omega ^{3}=xdy+dz.
\end{equation*}%
are given by%
\begin{equation}
\left \{
\begin{array}{l}
\nabla _{e_{1}}e_{1}=0 \\
\nabla _{e_{1}}e_{2}=\frac{\lambda }{2}e_{3} \\
\nabla _{e_{1}}e_{3}=-\frac{\lambda }{2}e_{2}%
\end{array}%
\begin{array}{l}
\nabla _{e_{2}}e_{1}=-\frac{\lambda }{2}e_{3} \\
\nabla _{e_{2}}e_{2}=0 \\
\nabla _{e_{2}}e_{3}=\frac{\lambda }{2}e_{1}%
\end{array}%
\begin{array}{l}
\nabla _{e_{3}}e_{1}=-\frac{\lambda }{2}e_{2} \\
\nabla _{e_{3}}e_{2}=\frac{\lambda }{2}e_{1} \\
\nabla _{e_{3}}e_{3}=0%
\end{array}%
\right.  \label{3}
\end{equation}%
The Lie bracket of the base $\left( e_{i}\right) _{i=\overline{1,3}}$ are
given by the following identities%
\begin{equation}
\lbrack e_{1},e_{2}]=\lambda e_{3};\  \  \ [e_{1},e_{3}]=[e_{2},e_{3}]=0
\end{equation}%
The Lie algebra of Killing vector field of $(\mathbb{H}_{3},g)$ is generated
by the following killing vectors%
\begin{eqnarray*}
\mathbf{K}_{1} &=&\partial z,\  \  \mathbf{K}_{2}=\partial y,\  \  \mathbf{K}%
_{3}=\partial x-y\partial z\ and \\
\mathbf{K}_{4} &=&\lambda ^{2}y\partial x-x\partial y-\frac{1}{2}(\lambda
^{2}y^{2}-x^{2})\partial z,
\end{eqnarray*}%
and using the Eq.(\ref{2}), we can rewrite the killing vectors in the base $%
\left( e_{i}\right) _{i=\overline{1,3}}$ as%
\begin{eqnarray}
\mathbf{K}_{1} &=&e_{3},\  \  \mathbf{K}_{2}=e_{1}+xe_{3},\  \  \  \  \mathbf{K}%
_{3}=\frac{1}{\lambda }e_{2}-ye_{3}\ and  \label{3.1} \\
\  \mathbf{K}_{4} &=&-xe_{1}+\lambda ye_{2}-\frac{1}{2}(\lambda
^{2}y^{2}-3x^{2})e_{3}.  \notag
\end{eqnarray}%
for more detail see (\cite{BZ}).

On the other hand, for the contact form%
\begin{equation}
\eta =xdy+dz=\omega ^{3}  \label{3.12}
\end{equation}%
and the $(1,1)$-tensor $\varphi $ given in base $\left( e_{i}\right) _{i=%
\overline{1,3}}$ by%
\begin{equation}
\varphi \left( e_{1}\right) =e_{2};\  \varphi \left( e_{2}\right) =-e_{1}\
and\  \varphi \left( e_{3}\right) =0  \label{3.11}
\end{equation}%
we have%
\begin{eqnarray*}
\eta \left( e_{3}\right) &=&1;\  \varphi ^{2}\left( X\right) =-X+\eta \left(
X\right) e_{3}\ and \\
g(\varphi \left( X\right) ,\varphi \left( Y\right) ) &=&g(X,Y)+\eta (X)\eta
(Y)
\end{eqnarray*}%
for any $X,Y\in \chi (\mathbb{H}_{3})$ and $\xi =e_{3}$, the Heisenberg
group $(\mathbb{H}_{3},\varphi ,\xi ,\eta ,g)$ is an almost contact
manifold. Moreover, we have
\begin{equation}
d\eta (X,Y)=g(X,\varphi \left( Y\right) )  \label{3.2}
\end{equation}%
then $(\mathbb{H}_{3},\varphi ,\xi ,\eta ,g)$ is a contact manifold and the
fundamental $2$-form $d\eta $ is closed and hence it defines a magnetic
field. For more detail see (\cite{f}, \cite{Ht}, \cite{cdpt}, \cite{m}).

\section{Killing magnetic curves in $\mathbb{H}_{3}$}

In this section, we have used two computer softwares (Wolfram Mathematica
and Scientific Workplace) to solve some differential systems and for curve
figures.

Let $\gamma (t)=(x(t),y(t),z(t)):I\subset \mathbb{R\rightarrow }\left(
\mathbb{H}_{3},g\right) $ be a regular curve. Its speed curve is
\begin{equation}
\gamma ^{\prime }(t)=\mathbf{t}=(x^{\prime }(t),y^{\prime }(t),z^{\prime
}(t))  \label{5}
\end{equation}%
from the Eq.(\ref{2}), the speed vector $\mathbf{t}$ is expressed in the
base $\left( e_{i}\right) _{i=\overline{1,3}}$ as%
\begin{equation}
\mathbf{t}=y^{\prime }e_{1}+\frac{x^{\prime }}{\lambda }e_{2}+\left(
z^{\prime }+xy^{\prime }\right) e_{3}  \label{6}
\end{equation}%
Using the connection formulas given in the Eq.(\ref{3}), the covariant
derivative of the speed vector $\mathbf{t}$ is%
\begin{equation}
\nabla _{\mathbf{t}}\mathbf{t=}\left( y^{\prime \prime }+x^{\prime }\left(
z^{\prime }+xy^{\prime }\right) \right) e_{1}+\left( \frac{x^{\prime \prime }%
}{\lambda }+\lambda y^{\prime }\left( z^{\prime }+xy^{\prime }\right)
\right) e_{2}+\left( z^{\prime }+xy^{\prime }\right) ^{\prime }e_{3}
\label{7}
\end{equation}%
We know that the equation of geodesics is a particular case of the Lorentz
equation given in Eq.(\ref{0.1}) when the Lorentz force vanishes, the
geodesics is a particular magnetic trajectories. A simple computation of the
geodesic equation $\nabla _{\mathbf{t}}\mathbf{t}=0,$ give a differential
equations system%
\begin{equation*}
S_{G}:\left \{
\begin{array}{l}
y^{\prime \prime }+x^{\prime }\left( z^{\prime }+xy^{\prime }\right) =0 \\
x^{\prime \prime }+\lambda ^{2}y^{\prime }\left( z^{\prime }+xy^{\prime
}\right) =0 \\
\left( z^{\prime }+xy^{\prime }\right) ^{\prime }=0%
\end{array}%
\right.
\end{equation*}%
where its general solutions are explicit formulas of geodesic curves $\gamma
_{G}$ in $(\mathbb{H}_{3},g)$ given by the following proposition.

\begin{proposition}
The parametric equations of the geodesic curves $\gamma _{G}$ in $(\mathbb{H}%
_{3},g)$\ is given by%
\begin{equation}
\begin{array}{l}
i.\  \gamma _{G}(t)=\left \{
\begin{array}{l}
x\left( t\right) =c_{1}t+c_{2} \\
y\left( t\right) =c_{3}t+c_{4} \\
z(t)=\frac{c_{1}c_{3}}{2}t^{2}+c_{2}c_{3}t+c_{5}%
\end{array}%
\right. \text{ or} \\
ii.\  \gamma _{G}(t)=\left \{
\begin{array}{l}
x\left( t\right) =-\frac{c_{1}}{c}e^{c\lambda t}-\frac{c_{2}}{c}e^{-c\lambda
t}+c_{3} \\
y\left( t\right) =\frac{c_{1}}{c\lambda }e^{c\lambda t}-\frac{c_{2}}{%
c\lambda }e^{-c\lambda t}+c_{4} \\
z(t)=\frac{2c_{1}c_{2}+c^{2}}{c}t-\frac{c_{3}}{\lambda c}\left(
c_{1}e^{ct\lambda }-c_{2}e^{-ct\lambda }\right) +\left( \frac{%
c_{1}^{2}-c_{2}^{2}}{2c^{2}\lambda }\right) e^{-2c\lambda t}%
\end{array}%
\right.%
\end{array}
\label{7.01}
\end{equation}%
where $c_{\overline{1,4}}$ are a real constants.
\end{proposition}

\subsection{$\mathbf{K}_{1}$- magnetic curves\label{ss1}}

We consider $\mathbf{K}_{1}$-magnetic curves which correspond to the Killing
vector field $\mathbf{K}_{1}=e_{3}=\partial z.$

Firstly, we have the product vector%
\begin{equation}
\mathbf{K}_{1}\times \mathbf{t}=-\frac{x^{\prime }}{\lambda }e_{1}+y^{\prime
}e_{2}  \label{7.1}
\end{equation}%
We can also obtain the Eq.(\ref{7.1}) in the following way.\newline
Because $\mathbf{K}_{1}=\xi $ given from the almost contact structure $(%
\mathbb{H}_{3},\varphi ,\xi ,\eta ,g)$, the magnetic fields $F_{\mathbf{K}%
_{1}}$ corresponding to $\mathbf{K}_{1}$ is exactly the magnetic fields $F$
associated to \textsc{Lorentz} force which presented \ by the skew-symmetric
$(1,1)$-tensor $\varphi $ given by the Eq.(\ref{3.11}).

Thus, we find the right-hand side of the relation given in Eq.(\ref{7.1})%
\begin{eqnarray*}
\varphi \left( \mathbf{t}\right) &=&\varphi \left( y^{\prime }e_{1}+\frac{%
x^{\prime }}{\lambda }e_{2}+\left( z^{\prime }+xy^{\prime }\right)
e_{3}\right) \\
&=&y^{\prime }e_{2}-\frac{x^{\prime }}{\lambda }e_{1}=\mathbf{K}_{1}\times
\mathbf{t}
\end{eqnarray*}

Now, to find the explicit formulas of $\mathbf{K}_{1}$-magnetic curves, we
must solve the differential equations system denoted by $(S_{1})$, given
from the Eqs.(\ref{7}, \ref{7.1}) and \textsc{Lorentz} equation%
\begin{equation*}
\nabla _{\mathbf{t}}\mathbf{t}=\mathbf{K}_{1}\times \mathbf{t}
\end{equation*}%
The system $(S_{1})$ is
\begin{equation*}
S_{1}:\left \{
\begin{array}{l}
y^{\prime \prime }+x^{\prime }\left( \left( z^{\prime }+xy^{\prime }\right) +%
\frac{1}{\lambda }\right) =0 \\
x^{\prime \prime }+y^{\prime }\left( \lambda ^{2}\left( z^{\prime
}+xy^{\prime }\right) -\lambda \right) =0 \\
\left( z^{\prime }+xy^{\prime }\right) ^{\prime }=0%
\end{array}%
\right.
\end{equation*}%
After integrating the third equation $\left( S_{1_{3}}\right) $ and replaced
them in the equations $\left( S_{1_{1,2}}\right) $, the system $(S)$ turns to%
\begin{equation*}
S_{1}:\left \{
\begin{array}{l}
y^{\prime \prime }+x^{\prime }\left( c+\frac{1}{\lambda }\right) =0 \\
x^{\prime \prime }+y^{\prime }\lambda \left( \lambda c-1\right) =0 \\
z^{\prime }+xy^{\prime }=c\  \ (constant)%
\end{array}%
\right.
\end{equation*}%
the general solution of the equations $(S_{1_{1,2}})$ is%
\begin{equation}
\left \{
\begin{array}{c}
x\left( t\right) =\frac{\lambda }{\lambda c+1}\left( c_{1}e^{\sqrt{\lambda
^{2}c^{2}-1}t}+c_{2}e^{-\sqrt{\lambda ^{2}c^{2}-1}t}\right) +c_{3} \\
y\left( t\right) =\frac{1}{\sqrt{\lambda ^{2}c^{2}-1}}\left( c_{1}e^{\sqrt{%
\lambda ^{2}c^{2}-1}t}-c_{2}e^{-\sqrt{\lambda ^{2}c^{2}-1}t}\right) +c_{4}%
\end{array}%
\right.  \label{8}
\end{equation}%
where $c_{\overline{1,4}}$ are a real constants, $\left \vert \lambda
c\right \vert >1$ (if $\left \vert \lambda c\right \vert <1$ the system
don't admit a real solutions) and $\lambda c\neq \pm 1$ (we studied these
cases $\lambda c=\pm 1$ later).

Substituting the Eq.(\ref{8}) in the equation $(S_{1_{3}})$ and by an
integration with respect to $t$, we have%
\begin{equation*}
z(t)=\left( c+\frac{2\lambda }{\lambda c+1}c_{1}c_{2}\right) t+\frac{1}{%
\sqrt{\lambda ^{2}c^{2}-1}}\left(
\begin{array}{c}
c_{1}c_{3}e^{-t\sqrt{\lambda ^{2}c^{2}-1}}-c_{2}c_{3}e^{t\sqrt{\lambda
^{2}c^{2}-1}}+ \\
\frac{\lambda }{2\left( \lambda c+1\right) }\left( c_{1}^{2}e^{-2t\sqrt{%
\lambda ^{2}c^{2}-1}}-c_{2}^{2}e^{2t\sqrt{\lambda ^{2}c^{2}-1}}\right)%
\end{array}%
\right) +c_{5}
\end{equation*}%
where $c_{5}$ is real constant.

If $c=\dfrac{1}{\lambda },$ we rewrite the system $(S_{1})$ as%
\begin{equation*}
S_{1}:\left \{
\begin{array}{l}
y^{\prime \prime }+x^{\prime }\frac{2}{\lambda }=0 \\
x^{\prime \prime }=0 \\
z^{\prime }+xy^{\prime }=\dfrac{1}{\lambda }%
\end{array}%
\right.
\end{equation*}%
Its general solution is%
\begin{equation*}
\left \{
\begin{array}{l}
x\left( t\right) =c_{1}t+c_{2}, \\
y\left( t\right) =\frac{-c_{1}}{\lambda }t^{2}+c_{3}t+c_{4} \\
z\left( t\right) =\frac{2c_{1}^{2}}{3\lambda }t^{3}+c_{1}\left( \frac{c_{2}}{%
\lambda }-\frac{c_{3}}{2}\right) t^{2}+\left( \frac{1}{\lambda }%
-c_{2}c_{3}\right) t+c_{5}%
\end{array}%
\right.
\end{equation*}%
where $c_{\overline{1,5}}$ are a real constants.

If $c=\dfrac{-1}{\lambda },$ the system $(S_{1})$ is%
\begin{equation*}
S_{1}:\left \{
\begin{array}{l}
y^{\prime \prime }=0 \\
x^{\prime \prime }-2y^{\prime }\lambda =0 \\
z^{\prime }+xy^{\prime }=\dfrac{-1}{\lambda }%
\end{array}%
\right.
\end{equation*}%
with a general solution%
\begin{equation*}
\left \{
\begin{array}{l}
x\left( t\right) =c_{1}\lambda t^{2}+c_{2}t+c_{3}, \\
y\left( t\right) =c_{1}t+c_{4} \\
z\left( t\right) =\frac{c_{1}c_{2}}{2}t^{2}+\left( \dfrac{-1}{\lambda }%
+c_{3}c_{1}-\frac{c_{1}^{2}}{3}\lambda \right) t+c_{5}%
\end{array}%
\right.
\end{equation*}

Now, we can present the following theorem.

\begin{theorem}
\label{TK1}The explicit formulae of all $\mathbf{K}_{1}$-magnetic curves in $%
(\mathbb{H}_{3},g)$ are the space curves given by parametric equations:%
\newline
1.\newline
$\gamma (t)=\left(
\begin{array}{l}
x\left( t\right) =c_{1}t+c_{2}, \\
y\left( t\right) =\frac{-c_{1}}{\lambda }t^{2}+c_{3}t+c_{4} \\
z\left( t\right) =\frac{2c_{1}^{2}}{3\lambda }t^{3}+c_{1}\left( \frac{c_{2}}{%
\lambda }-\frac{c_{3}}{2}\right) t^{2}+\left( \frac{1}{\lambda }%
-c_{2}c_{3}\right) t+c_{5}%
\end{array}%
\right) $ or\newline
2. \newline
$\gamma (t)=\left(
\begin{array}{l}
x\left( t\right) =c_{1}\lambda t^{2}+c_{2}t+c_{3}, \\
y\left( t\right) =c_{1}t+c_{4} \\
z\left( t\right) =\frac{c_{1}c_{2}}{2}t^{2}+\left( \tfrac{-1}{\lambda }%
+c_{3}c_{1}-\frac{c_{1}^{2}}{3}\lambda \right) t+c_{5}%
\end{array}%
\right) $ or \newline
3.\newline
$\gamma (t)=\left(
\begin{array}{l}
x\left( t\right) =\frac{\lambda }{\lambda c+1}\left( c_{1}e^{\sqrt{\lambda
^{2}c^{2}-1}t}+c_{2}e^{-\sqrt{\lambda ^{2}c^{2}-1}t}\right) +c_{3} \\
y\left( t\right) =\frac{1}{\sqrt{\lambda ^{2}c^{2}-1}}\left( c_{1}e^{\sqrt{%
\lambda ^{2}c^{2}-1}t}-c_{2}e^{-\sqrt{\lambda ^{2}c^{2}-1}t}\right) +c_{4}
\\
z\left( t\right) =\left( c+\frac{2\lambda c_{1}c_{2}}{\lambda c+1}\right) t+%
\frac{1}{\sqrt{\lambda ^{2}c^{2}-1}}\left(
\begin{array}{c}
c_{3}\left( c_{1}e^{-t\sqrt{\lambda ^{2}c^{2}-1}}-c_{2}e^{t\sqrt{\lambda
^{2}c^{2}-1}}\right) + \\
\frac{\lambda }{2\left( \lambda c+1\right) }\left( c_{1}^{2}e^{-2t\sqrt{%
\lambda ^{2}c^{2}-1}}-c_{2}^{2}e^{2t\sqrt{\lambda ^{2}c^{2}-1}}\right)%
\end{array}%
\right) +c_{5}%
\end{array}%
\right) $\newline
where in (3), $\left \vert \lambda c\right \vert >1$ and $c,c_{\overline{1,5}%
}$ are a real constants.
\end{theorem}

\begin{corollary}
The curve presented in assertion (1) of the Theorem (\ref{TK1}) for $c_{1}=0$
and $\lambda =\frac{1}{2c_{2}c_{3}}$ is a geodesic $\mathbf{K}_{1}$-magnetic
curve in $(\mathbb{H}_{3},g).$
\end{corollary}

\begin{proof}
It's a direct consequence from the geodesic curve formulas in Eq.(\ref{7.01}%
) an the Theorem (\ref{TK1}).
\end{proof}

We present in $(\mathbb{R}_{3},g_{euc})$ some examples of $\mathbf{K}_{1}$%
-magnetic curve in $(\mathbb{H}_{3},g)$ in following figures drawing in $(%
\mathbb{R}_{3},g_{euc})$. \newline
For $c_{1,2.3}=1;\  \ c_{4,5}=0$ and$\  \lambda =1$, the figure 1 at left
side, present the $\mathbf{K}_{1}$-magnetic curve for the first assertion in
$(\mathbb{H}_{3},g)$.
\begin{figure}[h]
\centering
\includegraphics[width=2in,height=1.6in]{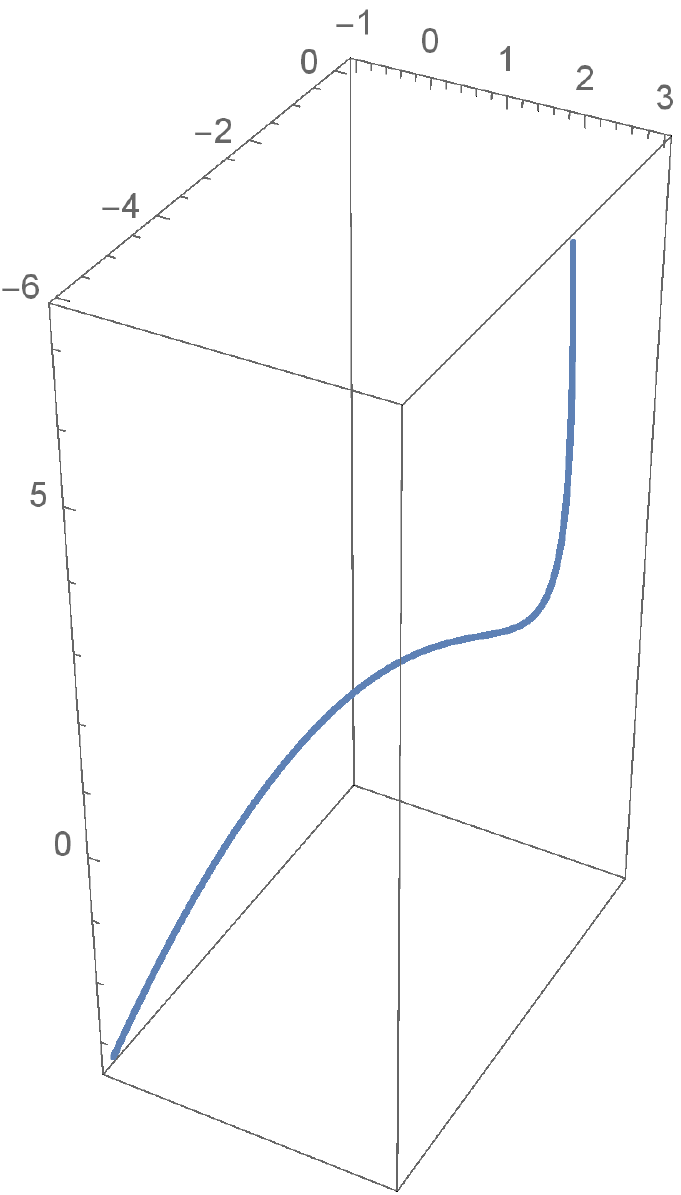} %
\includegraphics[width=2in,height=1.6in]{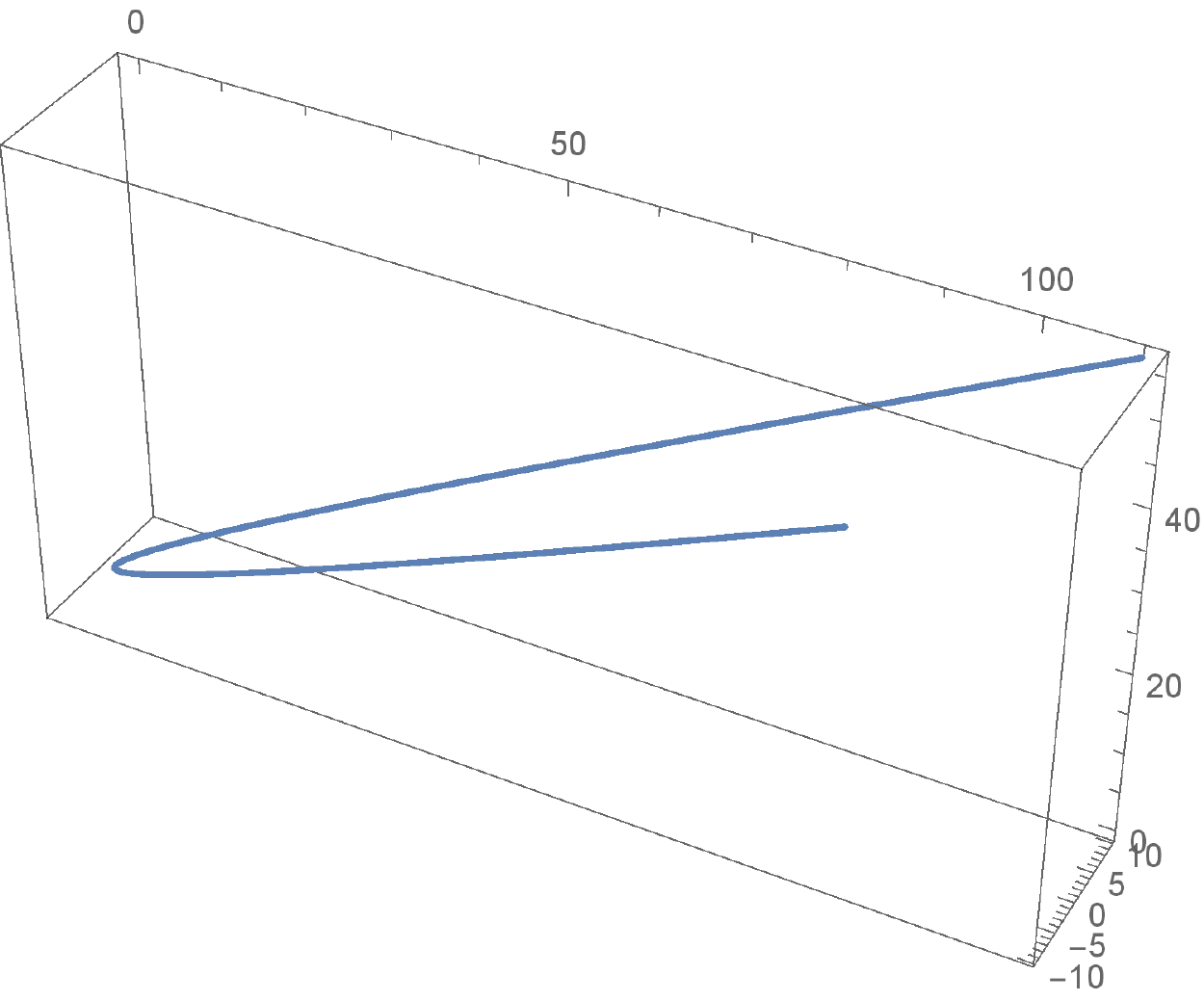}
\caption{$\mathbf{K}_{1}$-magnetic curve}
\end{figure}
\newline
For $c_{1,2}=\lambda =1\ $and$\ c_{3,4,5}=0,$ the figure 1 at right side,
present the $\mathbf{K}_{1}$-magnetic curve for the assertion (2) in $(%
\mathbb{H}_{3},g)$.
\begin{figure}[h]
\centering
\includegraphics[width=2in,height=1.6in]{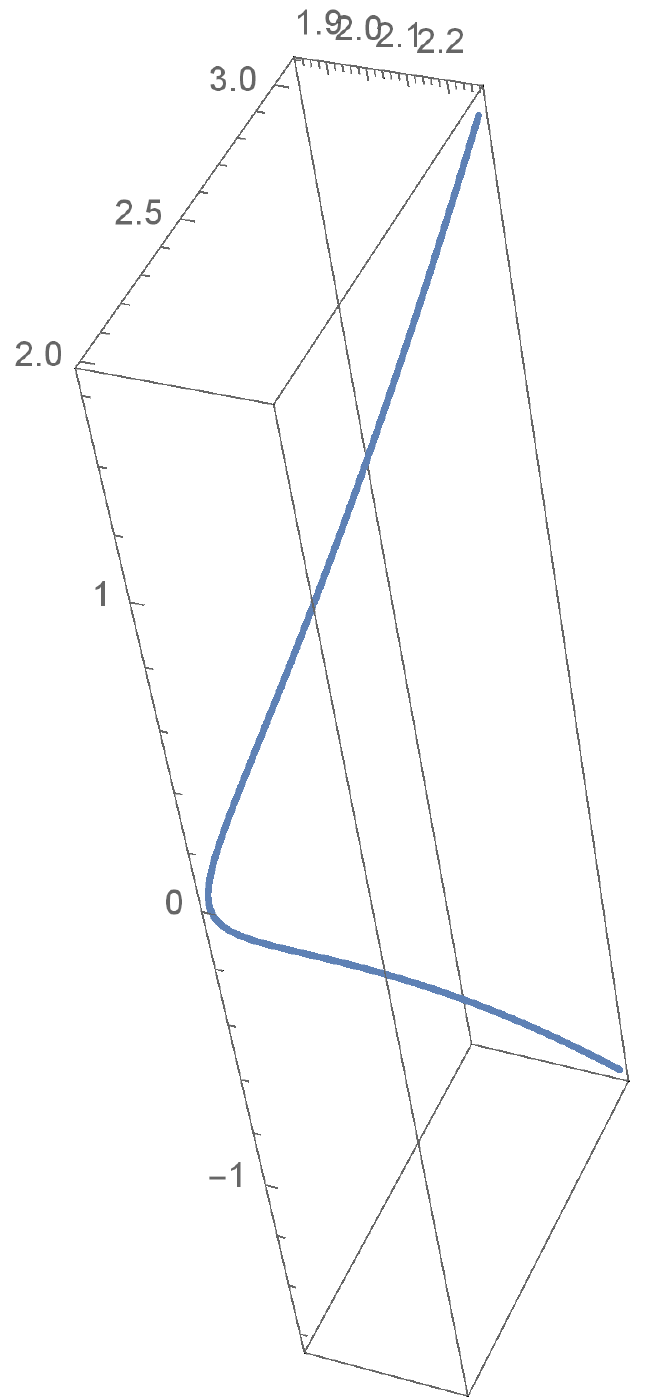}
\caption{$\mathbf{K}_{1}$-magnetic curve}
\end{figure}
\newline
For $c_{1,2,3}=1;\  \ c_{4,5}=0$ and$\  \lambda =1;\ c=\sqrt{2}$ (i.e. $%
\left \vert \lambda c\right \vert >1$)$,$ the figure 2 present the $\mathbf{K}%
_{1}$-magnetic curve for the assertion (3) in $(\mathbb{H}_{3},g)$.

\subsection{$\mathbf{K}_{2}$-magnetic curves\label{ss2}}

Similarly, as in subsection (\ref{ss1}), we have%
\begin{equation}
\mathbf{K}_{2}\times \mathbf{t}=-\frac{x^{\prime }x}{\lambda }e_{1}+\left(
xy^{\prime }-\left( z^{\prime }+xy^{\prime }\right) \right) e_{2}+\frac{%
x^{\prime }}{\lambda }e_{3}  \label{9}
\end{equation}

We can find a same formula as the Eq.(\ref{9}) using the Eqs.(\ref{0.02}, %
\ref{0.03} and \ref{3.12}) to determine the magnetic fields $F_{\mathbf{K}%
_{2}}$and associated skew-symmetric $(1,1)$-tensor $\varphi .$

The $\mathbf{K}_{2}$-magnetic curves are the solution of the differential
equations system
\begin{equation*}
S_{2}:\left \{
\begin{array}{l}
x^{\prime }\left( z^{\prime }+xy^{\prime }\right) =-y^{\prime \prime }-\frac{%
xx^{\prime }}{\lambda } \\
\left( \lambda y^{\prime }+1\right) \left( z^{\prime }+xy^{\prime }\right)
=xy^{\prime }-\frac{x^{\prime \prime }}{\lambda } \\
\left( z^{\prime }+xy^{\prime }\right) ^{\prime }=\frac{x^{\prime }}{\lambda
}%
\end{array}%
\right.
\end{equation*}%
given from the Eqs.(\ref{7}, \ref{9}) and \textsc{Lorentz} equation%
\begin{equation*}
\nabla _{\mathbf{t}}\mathbf{t}=\mathbf{K}_{2}\times \mathbf{t}
\end{equation*}%
after integrating the equation $\left( S_{2_{3}}\right) ,$ we have
\begin{equation}
z^{\prime }+xy^{\prime }=\frac{x}{\lambda }+c\text{ }  \label{9.1}
\end{equation}%
here $c$ is a arbitrary real constant. Substituting the last equation in the
equations $\left( S_{2_{2,3}}\right) ,$ we get the system%
\begin{equation*}
\overline{S}_{2}:\left \{
\begin{array}{l}
x^{\prime }\left( \frac{2x}{\lambda }+c\right) =-y^{\prime \prime } \\
\left( \lambda y^{\prime }+1\right) \left( \frac{x}{\lambda }+c\right)
=xy^{\prime }-\frac{x^{\prime \prime }}{\lambda }%
\end{array}%
\right.
\end{equation*}%
and%
\begin{equation*}
\overline{S}_{2}:\left \{
\begin{array}{l}
y^{\prime }=-\frac{x^{2}}{\lambda }-xc \\
x^{\prime \prime }-cx^{2}\lambda +x\left( 1-c^{2}\lambda ^{2}\right)
+c\lambda =0.%
\end{array}%
\right.
\end{equation*}%
it is very difficult to find exact solutions of the system ($\overline{S}%
_{2} $) for every constant $c$. However, without loss of generality, we can
assume that $c=0,$ the solution can be given as%
\begin{equation*}
\left \{
\begin{array}{l}
x(t)=c_{1}\cos t-c_{2}\sin t \\
y(t)=\frac{1}{2\lambda }c_{1}c_{2}\cos 2t+\frac{1}{4\lambda }\left(
c_{1}^{2}-c_{2}^{2}\right) \sin 2t+\frac{1}{2\lambda }\left(
c_{1}^{2}+c_{2}^{2}\right) t%
\end{array}%
\right.
\end{equation*}%
Substituting the solutions $x(t)$ and $y(t)$ in Eq.(\ref{9.1}) and by an
integration, we have the following solution
\begin{equation*}
z(t)=\frac{1}{\lambda }\left(
\begin{array}{c}
\frac{1}{4}c_{2}\left( c_{1}^{2}-\frac{1}{3}c_{2}^{2}\right) \cos 3t+\frac{1%
}{4}c_{1}\left( \frac{1}{3}c_{1}^{2}-c_{2}^{2}\right) \sin 3t \\
+\left( 1+\frac{3}{4}\left( c_{2}^{2}+c_{1}^{2}\right) \right) \left(
c_{2}\cos t+c_{1}\sin t\right)%
\end{array}%
\right)
\end{equation*}

Then we have the following theorem:

\begin{theorem}
\label{TK2}The space curves given by parametric equations
\begin{equation*}
\gamma (t)=\left \{
\begin{array}{l}
x(t)=c_{1}\cos t-c_{2}\sin t \\
y(t)=\frac{1}{2\lambda }c_{1}c_{2}\cos 2t+\frac{1}{4\lambda }\left(
c_{1}^{2}-c_{2}^{2}\right) \sin 2t+\frac{1}{2\lambda }\left(
c_{1}^{2}+c_{2}^{2}\right) t \\
z(t)=\frac{1}{\lambda }\left(
\begin{array}{c}
\frac{1}{4}c_{2}\left( c_{1}^{2}-\frac{1}{3}c_{2}^{2}\right) \cos 3t+\frac{1%
}{4}c_{1}\left( \frac{1}{3}c_{1}^{2}-c_{2}^{2}\right) \sin 3t \\
+\left( 1+\frac{3}{4}\left( c_{2}^{2}+c_{1}^{2}\right) \right) \left(
c_{2}\cos t+c_{1}\sin t\right)%
\end{array}%
\right)%
\end{array}%
\right.
\end{equation*}%
are $\mathbf{K}_{2}$-Killing magnetic curves in $(\mathbb{H}_{3},g),$ where $%
c_{1},\ c_{2}$ are a real constants.
\end{theorem}

\begin{corollary}
There is no $\mathbf{K}_{2}$-magnetic curves od the space curves of the
Theorem (\ref{TK2}) which is a geodesic in $(\mathbb{H}_{3},g).$
\end{corollary}

In $(\mathbb{R}_{3},g_{euc}),$ we present an example of $\mathbf{K}_{2}$%
-magnetic curve in $(\mathbb{H}_{3},g)$ in figure $3$ with $c_{1}=2c_{2}=2$
and $\lambda =1.$
\begin{figure}[h]
\centering
\includegraphics[width=2in,height=1.6in]{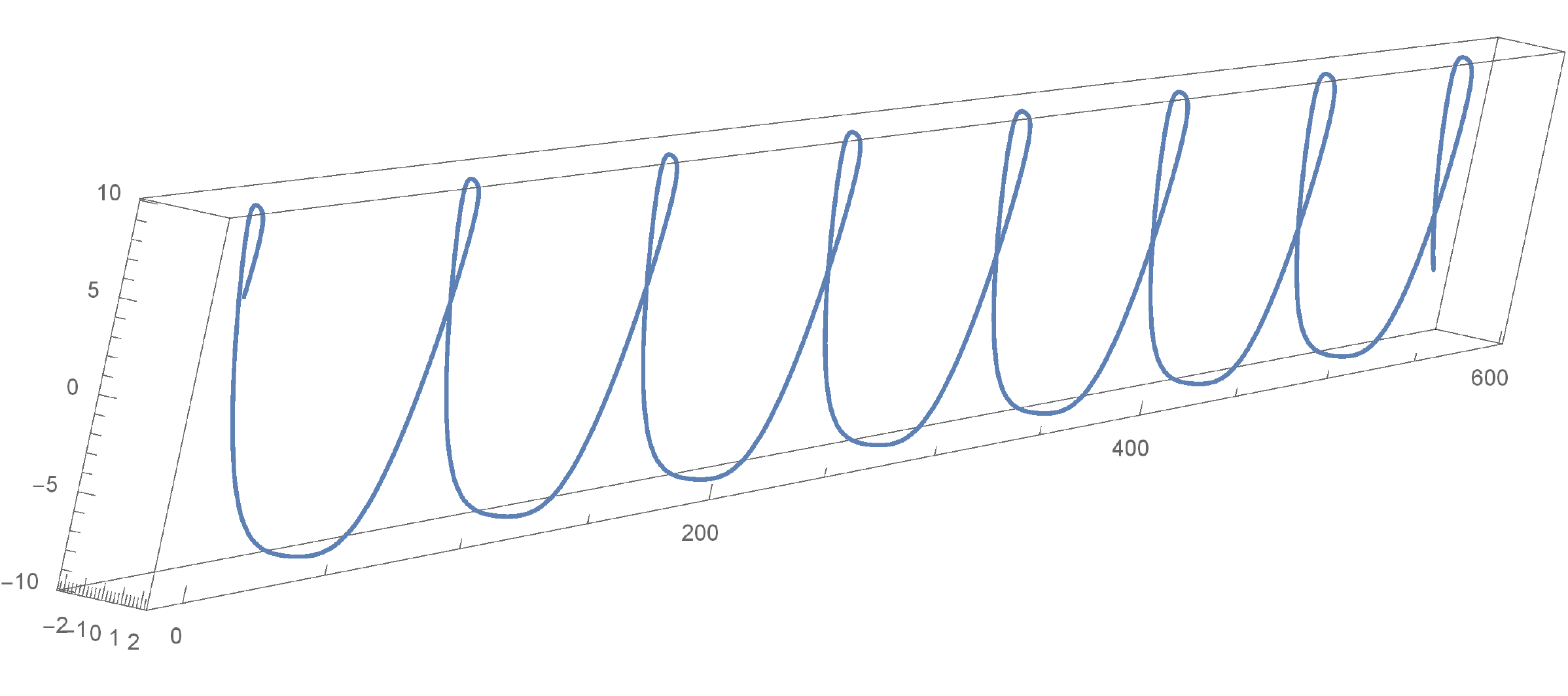}
\caption{$\mathbf{K}_{2}$-magnetic curve in $(\mathbb{H}_{3},g)$ presented
in $(\mathbb{R}_{3},g_{euc})$}
\end{figure}

\subsection{$\mathbf{K}_{3}$-magnetic curves\label{ss3}}

As above in subsections \ref{ss1} and \ref{ss2}, we have the product vector
\begin{equation*}
\mathbf{K}_{3}\times \mathbf{t=}\frac{1}{\lambda }\left( yx^{\prime
}+z^{\prime }+xy^{\prime }\right) e_{1}-yy^{\prime }e_{2}-\frac{1}{\lambda }%
y^{\prime }e_{3}
\end{equation*}%
Also, we can find a same formula as the Eq.(\ref{9}) using the Eqs.(\ref%
{0.02}, \ref{0.03} and \ref{3.12}) to determine the magnetic fields $F_{%
\mathbf{K}_{2}}$and the associated skew-symmetric $(1,1)$-tensor $\varphi .$

Using the Eqs(\ref{7}, \ref{9}) and%
\begin{equation*}
\nabla _{\mathbf{t}}\mathbf{t}=\mathbf{K}_{3}\times \mathbf{t}
\end{equation*}%
we have the differential equations system $\left( S_{3}\right) $%
\begin{equation*}
S_{3}:\left \{
\begin{array}{c}
y^{\prime \prime }+x^{\prime }\left( z^{\prime }+xy^{\prime }\right) =\frac{1%
}{\lambda }\left( z^{\prime }+xy^{\prime }\right) +y\frac{x^{\prime }}{%
\lambda } \\
\frac{x^{\prime \prime }}{\lambda }+\lambda y^{\prime }\left( z^{\prime
}+xy^{\prime }\right) =-yy^{\prime } \\
\left( z^{\prime }+xy^{\prime }\right) ^{\prime }=-\frac{1}{\lambda }%
y^{\prime }%
\end{array}%
\right.
\end{equation*}%
\newline
The integration of the third equation $\left( S_{3_{3}}\right) $, give%
\begin{equation}
z^{\prime }+xy^{\prime }=-\frac{1}{\lambda }y+c  \label{10}
\end{equation}%
substituting the last equation in equations $\left( S_{3_{1,2}}\right) ,$ we
have the differential equations system%
\begin{equation*}
\overline{S}_{3}:\left \{
\begin{array}{c}
y^{\prime \prime }+2\lambda cy^{2}+y\left( \frac{1}{\lambda ^{2}}-\lambda
^{2}c^{2}\right) -\frac{c}{\lambda }=0 \\
x^{\prime }=-\lambda ^{2}cy%
\end{array}%
\right.
\end{equation*}%
Analogously as the subsection \ref{ss2}, it is a true challenge to find
exact solutions for the system ($\overline{S}_{3}$). Therefore, we try to
solve'it in particular case $c=0$.

By integrating the system $\left( \overline{S}_{3}\right) $ in the case $%
c=0, $ we find the general solution%
\begin{equation*}
\left \{
\begin{array}{l}
x(t)=c_{1}t+c_{2} \\
y(t)=c_{3}\exp \left( \frac{\sqrt{2\lambda c_{1}-1}}{\lambda }t\right)
+c_{4}\exp \left( \frac{-\sqrt{2\lambda c_{1}-1}}{\lambda }t\right)%
\end{array}%
\right.
\end{equation*}%
Substituting the solutions $x(t)$ and $y(t)$ in the Eq.(\ref{10}), we get
\begin{equation*}
z(t)=%
\begin{array}{c}
\left( \frac{-c_{3}}{\sqrt{2\lambda c_{1}-1}}-c_{3}\left( c_{2}-\lambda c_{1}%
\sqrt{2\lambda c_{1}-1}+c_{1}t\right) \right) \exp \left( \frac{\sqrt{%
2\lambda c_{1}-1}}{\lambda }t\right) \\
+\left( \frac{c_{4}}{\sqrt{2\lambda c_{1}-1}}-c_{4}\left( c_{2}+\lambda c_{1}%
\sqrt{2\lambda c_{1}-1}+c_{1}t\right) \right) \exp \left( -\frac{\sqrt{%
2\lambda c_{1}-1}}{\lambda }t\right)%
\end{array}%
\end{equation*}%
where $c_{\overline{1,4}}$ are an arbitrary real constants.

\bigskip

Finally, we can present the following theorem.

\begin{theorem}
\label{TK3}The space curves given by parametric equations%
\begin{equation*}
\gamma (t)=\left(
\begin{array}{l}
x(t)=c_{1}t+c_{2} \\
y(t)=c_{3}\exp \left( \frac{\sqrt{2\lambda c_{1}-1}}{\lambda }t\right)
+c_{4}\exp \left( \frac{-\sqrt{2\lambda c_{1}-1}}{\lambda }t\right) \\
z(t)=%
\begin{array}{c}
\left( \frac{-c_{3}}{\sqrt{2\lambda c_{1}-1}}-c_{3}\left( c_{2}-\lambda c_{1}%
\sqrt{2\lambda c_{1}-1}+c_{1}t\right) \right) \exp \left( \frac{\sqrt{%
2\lambda c_{1}-1}}{\lambda }t\right) \\
+\left( \frac{c_{4}}{\sqrt{2\lambda c_{1}-1}}-c_{4}\left( c_{2}+\lambda c_{1}%
\sqrt{2\lambda c_{1}-1}+c_{1}t\right) \right) \exp \left( -\frac{\sqrt{%
2\lambda c_{1}-1}}{\lambda }t\right)%
\end{array}%
\end{array}%
\right)
\end{equation*}%
are $\mathbf{K}_{3}$-magnetic curves in $(\mathbb{H}_{3},g),$ where $c_{%
\overline{1,4}}$ are a real constants.
\end{theorem}

\bigskip

\begin{corollary}
There is no $\mathbf{K}_{3}$-magnetic curves of the space curves of the
Theorem (\ref{TK3}) which is a geodesic in $(\mathbb{H}_{3},g).$
\end{corollary}

The following figure present in $(\mathbb{R}_{3},g_{euc})$ an example of $%
\mathbf{K}_{3}$-magnetic curve in $(\mathbb{H}_{3},g)$ with $c_{\overline{1,4%
}}=1$ and $\lambda =1.$
\begin{figure}[h]
\centering
\includegraphics[width=2in,height=1.6in]{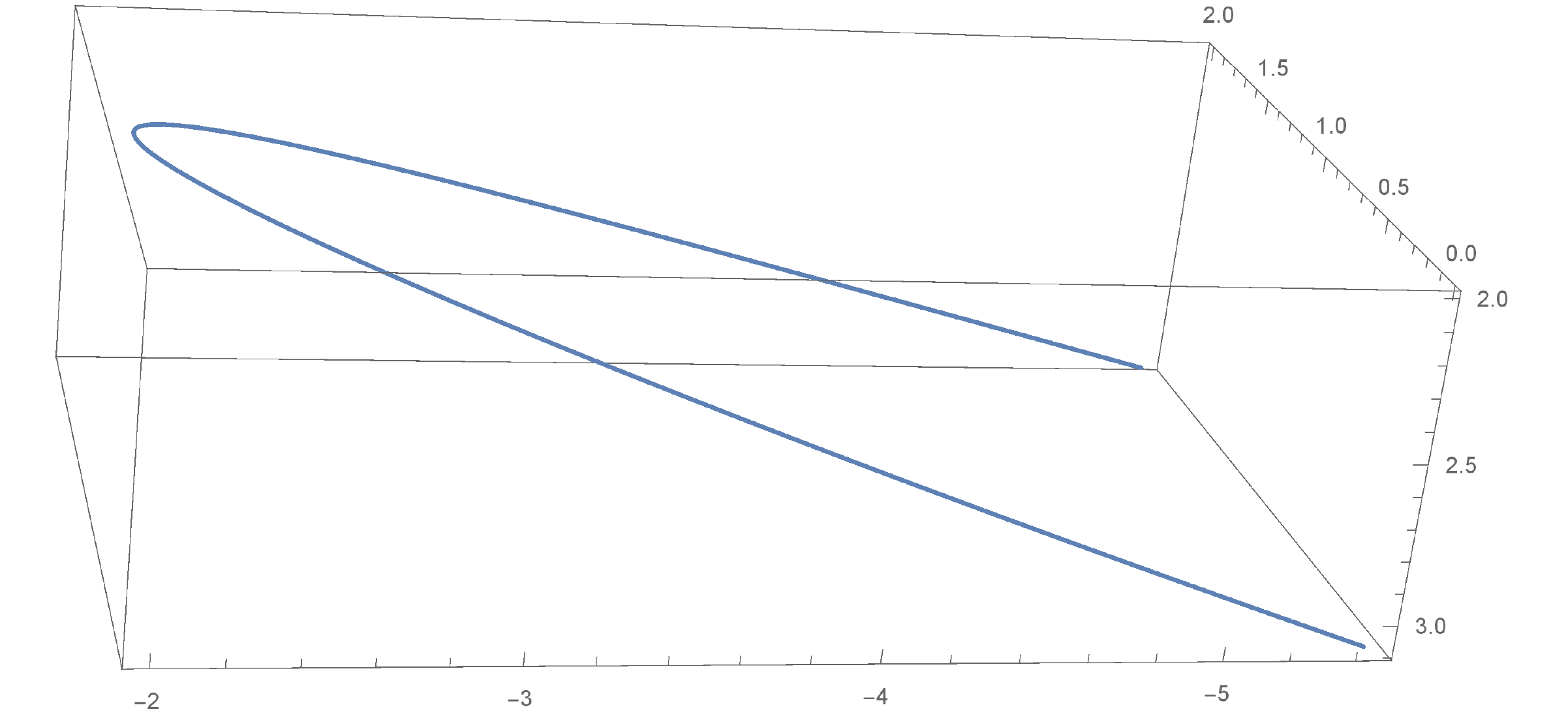}
\caption{$\mathbf{K}_{3}$-magnetic curve in $(\mathbb{H}_{3},g)$ presented
in $(\mathbb{R}_{3},g_{euc})$}
\end{figure}

\vspace{0.5in}

\subsection{$\mathbf{K}_{4}$-magnetic curves}

Finally, as subsections \ref{ss1}, \ref{ss2} and \ref{ss3}, we consider $%
\mathbf{K}_{4}$-magnetic curves which correspond to the Killing vector field
$\mathbf{K}_{4}.$

Firstly, we have the product vector

\begin{eqnarray}
\mathbf{K}_{4}\times \mathbf{t} &=&-\frac{1}{2\lambda }\left(
3x^{2}x^{\prime }-y^{2}x^{\prime }\lambda ^{2}-2y\lambda ^{2}\left(
z^{\prime }+xy^{\prime }\right) \right) e_{1}  \label{11} \\
&&-\frac{1}{2}\left( -3x^{2}y^{\prime }+y^{2}y^{\prime }\lambda
^{2}-2x\left( z^{\prime }+xy^{\prime }\right) \right) e_{2}-\frac{1}{\lambda
}\left( xx^{\prime }+yy^{\prime }\lambda ^{2}\right) e_{3}  \notag
\end{eqnarray}%
The same formula can be found as the Eq.(\ref{11}), using the Eqs.(\ref{0.02}%
, \ref{0.03} and \ref{3.12}) to determine the magnetic fields $F_{\mathbf{K}%
_{4}}$and the associated skew-symmetric $(1,1)$-tensor $\varphi .$

Using the Eqs(\ref{7}, \ref{11}) and \textsc{Lorentz} equation formula%
\begin{equation*}
\nabla _{\mathbf{t}}\mathbf{t}=\mathbf{K}_{4}\times \mathbf{t}
\end{equation*}%
we have the differential equations system $\left( S_{4}\right) $%
\begin{equation*}
S_{4}:\left \{
\begin{array}{l}
y^{\prime \prime }+x^{\prime }\left( z^{\prime }+xy^{\prime }\right) =\frac{3%
}{2\lambda }x^{2}x^{\prime }-\frac{\lambda }{2}y^{2}x^{\prime }-\lambda
y\left( z^{\prime }+xy^{\prime }\right) \\
x^{\prime \prime }+\lambda ^{2}y^{\prime }\left( z^{\prime }+xy^{\prime
}\right) =\frac{3\lambda }{2}x^{2}y^{\prime }-\frac{\lambda ^{3}}{2}%
y^{2}y^{\prime }+\lambda x\left( z^{\prime }+xy^{\prime }\right) \\
\left( z^{\prime }+xy^{\prime }\right) ^{\prime }=-\lambda yy^{\prime }-x%
\frac{x^{\prime }}{\lambda }%
\end{array}%
\right.
\end{equation*}%
the integration of the third equation $\left( S_{4_{3}}\right) ,$ give%
\begin{equation}
z^{\prime }+xy^{\prime }=-\frac{x^{2}}{2\lambda }-\frac{\lambda }{2}y^{2}+c
\label{11.1}
\end{equation}%
Substituting the last equation in the equations $\left( S_{4_{1,2}}\right) $%
, we get%
\begin{equation*}
\overline{S}_{4}:\left \{
\begin{array}{l}
y^{\prime \prime }-\frac{2}{\lambda }x^{2}x^{\prime }=\frac{1}{2}y\left(
y^{2}\lambda ^{2}+x^{2}\right) -c\left( y\lambda +x^{\prime }\right) \\
x^{\prime \prime }-2y^{\prime }x^{2}\lambda =-\frac{1}{2}x\left(
y^{2}\lambda ^{2}+x^{2}\right) +c\left( x\lambda -cy^{\prime }\lambda
^{2}\right)%
\end{array}%
\right.
\end{equation*}%
It is not essay to solve the last differential equations system $\left(
\overline{S}\right) $ or is not exactly solvable in general case, however we
can solve'it in particular case when $c=0$ and let $x(t)=y(t)$ and taking
account that $y^{2}\lambda ^{2}+x^{2}\neq 0,$ then $\left( \overline{S}%
_{4}\right) $ terns to%
\begin{equation*}
y^{\prime \prime }-\frac{1+\lambda ^{2}}{\lambda }y^{2}y^{\prime }=0
\end{equation*}%
by integrating the last equation with respect to $y$, we have
\begin{equation*}
y^{\prime }-\frac{1+\lambda ^{2}}{3\lambda }y^{3}=c_{1}
\end{equation*}%
where $c_{1}$ is a real constant. Without loss of generality, we can assume
that $c_{1}=0,$ then the non null solution is%
\begin{equation*}
y(t)=\pm \frac{1}{\sqrt{2}\sqrt{c_{2}-\left( \frac{1+\lambda ^{2}}{3\lambda }%
\right) t}}
\end{equation*}%
where $c_{2}$\ is\ a\ constant. From the Eq.(\ref{11.1}), we have%
\begin{equation*}
z(t)=\frac{3}{4}\ln \left \vert \left( \frac{1+\lambda ^{2}}{3\lambda }%
\right) t-c_{2}\right \vert \pm \frac{1}{4\left( c_{2}-\left( \frac{%
1+\lambda ^{2}}{3\lambda }\right) t\right) }+c_{3}
\end{equation*}%
where $c_{2},c_{3}$ are a real constants. then we have the theorem.

\begin{theorem}
\label{TK4}The space curves given by parametric equations%
\begin{equation*}
\gamma (t)=\left(
\begin{array}{l}
x(t)=\frac{1}{\sqrt{2}\sqrt{c_{1}-\left( \frac{1+\lambda ^{2}}{3\lambda }%
\right) t}} \\
y(t)=\pm \frac{1}{\sqrt{2}\sqrt{c_{1}-\left( \frac{1+\lambda ^{2}}{3\lambda }%
\right) t}} \\
z(t)=\frac{3}{4}\ln \left \vert \left( \frac{1+\lambda ^{2}}{3\lambda }%
\right) t-c_{1}\right \vert \pm \frac{1}{4\left( c_{1}-\left( \frac{%
1+\lambda ^{2}}{3\lambda }\right) t\right) }+c_{2}%
\end{array}%
\right)
\end{equation*}%
are $\mathbf{K}_{4}$-magnetic curves in $(\mathbb{H}_{3},g)$, where $%
c_{1},c_{2}$ are an arbitrary real constants.
\end{theorem}

\begin{corollary}
There is no $\mathbf{K}_{4}$-magnetic curves of the space curves of the
Theorem (\ref{TK4}) which is a geodesic in $(\mathbb{H}_{3},g).$
\end{corollary}

Finally, we present in the following figure in $(\mathbb{R}_{3},g_{euc}),$
an example of $\mathbf{K}_{4}$-magnetic curve in $(\mathbb{H}_{3},g)$ where $%
c_{1}=c_{2}=0$ and $\lambda =1.$
\begin{figure}[h]
\centering
\includegraphics[width=2in,height=1.6in]{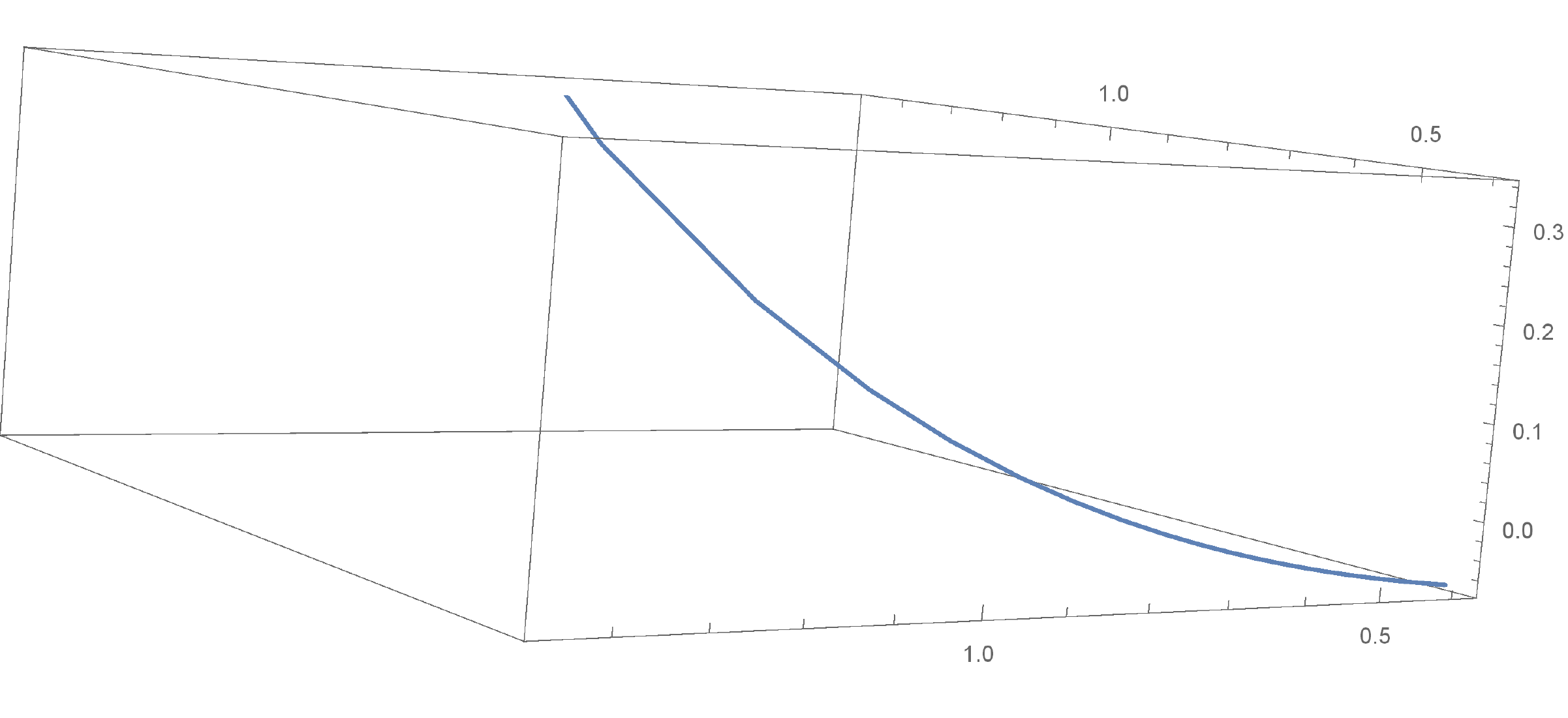}
\caption{$\mathbf{K}_{4}$-magnetic curve in $(\mathbb{H}_{3},g)$ presented
in $(\mathbb{R}_{3},g_{euc})$}
\end{figure}


\begin{thebibliography}{99}
\bibitem{br} M. Barros, A. Romero: Magnetic vortices, EPL 77 (2007), 34002.

\bibitem{bd} C. Bejan, S. L. Drut\u{a}-Romaniuc. Walker manifolds and
Killing magnetic curves, Diff. Geo. and its App. , 35 (2014), 106-116.
doi.org/10.1016/j.difgeo.2014.03.001.

\bibitem{BZ} W. Batat, A. Zaeim. On symmetries of the Heisenberg group,
arXiv:1710.04539v1.

\bibitem{cdpt} L. Capogna, D. Danielli, S.D. Pauls, J.T. Tyson. An
introduction to the Heisenberg group and the sub-Riemannian isoperimetric
problem. (Progress in Mathematics). Birkh\"{a}user Verlag, Basel (2007).

\bibitem{cmp} G. Calvaruso, M. I. Munteanu, A. Perrone. Killing magnetic
curves in three-dimensional almost paracontact manifolds, J. of Math. Ana.
and App., 426(1), (2015), 423-439. doi.org/10.1016/j.jmaa.2015.01.057

\bibitem{dim} S. L. Drut\u{a}-Romaniuc, J. Inoguchi, M. I. Munteanu and A.
I. Nistor: Magnetic curves in cosymplectic manifolds, Rep. Math. Phys. 78
(2016), 33.

\bibitem{dimn} S. L. Drut\u{a}-Romaniuc, J. Inoguchi, M. I. Munteanu, A. I.
Nistor: Magnetic curves in Sasakian manifolds, J. Nonlinear Math. Phys. 22
(2015), 428.

\bibitem{dm1} S. L. Drut\u{a}-Romaniuc, M. I. Munteanu. Killing magnetic
curves in a Minkowski 3-space, Nonlinear Analysis: Real World Appl. 14
(2013), 383.

\bibitem{dm} S. L. Drut\u{a}-Romaniuc and M. I. Munteanu, Magnetic curves
corresponding to Killing magnetic fields in $\mathbb{E}^{3}$, J. Math. Phys.
52 (2011), 113506.

\bibitem{e} Z. Erjavec. On Killing magnetic curves in $SL(2,\mathbb{R})$
geometry, Reports on mathematical physics, 84(3), (2019), 333-350.

\bibitem{ei} Z. Erjavec, Ji. Inoguchi. On Magnetic Curves in Almost
Cosymplectic Sol Space. Results Math 75, 113 (2020).
doi.org/10.1007/s00025-020-01235-y

\bibitem{ei2} Z. Erjavec, Ji. Inoguchi. Killing Magnetic Curves in Sol
Space. Math. Phys. Anal. Geom. 21, 15 (2018).
doi.org/10.1007/s11040-018-9272-6

\bibitem{f} C. Figueroa. The Gauss map of minimal graphs in the Hzisenberg
group, J. of Geo. and Sym. in Phys. 25, (2012), 1-21.

\bibitem{Ht} T. Hangan. Au sujet des flots riemanniens sur le groupe
nilpotent de Heisenberg. Rend. Circ. Mat. Palermo 35, 291--305 (1986).
doi.org/10.1007/BF02844738.

\bibitem{isc} Z. Iqbal, J. Sengupta, S. Chakraborty. Magnetic trajectories
corresponding to Killing magnetic fields in a three-dimensional warped
product 2020, Inter. J. of Geo. Meth. in Mod. Phys., 17(14), 2050212 (2020).
doi.org/10.1142/S0219887820502126

\bibitem{m} J. Milnor, Curvature of left invariant metrics on Lie groups,
Adv. Math. 21 (1976), 293-329.

\bibitem{mn} M. I. Munteanua, A. Nistor. The classification of Killing
magnetic curves in $\mathbb{S}^{2}\times \mathbb{R}$, J. of Geo. and Phys.,
62 (2012) 170--182. doi:10.1016/j.geomphys.2011.10.002

\bibitem{o} C. \"{O}zg\"{u}r. On magnetic curves in $3$-dimensional
Heisenberg group, Pro. of the Ins. of Math. and Mech., National Academy of
Sciences of Azerbaijan, 43(2), (2017), 278-286.
\end{thebibliography}
\end{document}